\documentclass[12pt]{article}

\usepackage{amsmath}
\usepackage{amsthm}
\usepackage{textcomp}

\usepackage{geometry}                
\usepackage[parfill]{parskip}    
\usepackage{graphicx}
\usepackage{amssymb,latexsym} 
\usepackage{epstopdf}

\newtheorem{theorem}{Theorem} 
\newtheorem{lemma}[theorem]{Lemma}   

\newtheorem{corollary}[theorem]{Corollary}

\newtheorem*{PF}{Perron-Frobenius Theorem}
\newtheorem*{mainTheorem}{Main Theorem}
\newtheorem*{thm}{Theorem}

\newcommand{\R}{\mathbb{R}}

\newcommand{\Z}{\mathbb{Z}}
\newcommand{\Q}{\mathbb{Q}}
\newcommand{\C}{\mathbb{C}}
\newcommand{\N}{\mathbb{N}}

\newcommand{\eps}{\epsilon}
\newcommand{\goesto}[1]{\xrightarrow{#1}}
\newcommand{\inv}{^{-1}}

\newcommand{\cross}{\times}

\newcommand{\Sp}{\mathrm{Sp}\left(2g\right)}

\newcommand{\Spq}{\mathrm{Sp} \left(2g, \Q \right)}
\newcommand{\SpL}{\mathrm{Sp} \left(2g, \Z, L\right)}
\newcommand{\SpJ}{\mathrm{Sp} \left(2g, \Z, J \right)}
\newcommand{\SpLq}{\mathrm{Sp} \left(2g, \Q, L \right)}
\newcommand{\Ug}{\mathrm{U}\left(g\right)}
\newcommand{\Sg}{\mathbf{S}^{2g-1}}
\newcommand{\Rg}{\mathbb{R}^{2g}}
\newcommand{\GL}{\mathrm{GL}}
\newcommand{\SL}{\mathrm{SL}}
\newcommand{\Og}{\mathrm{O} \left(2g \right)}

\title{Achievable spectral radii of symplectic Perron-Frobenius matrices}
\author{R. Ackermann}

\begin{document}

\maketitle

\begin{abstract}
A pseudo-Anosov surface automorphism $\phi$ has associated to it an algebraic unit $\lambda_\phi$ called the dilatation of $\phi$.  It is known that in many cases $\lambda_\phi$ appears as the spectral radius of a Perron-Frobenius matrix preserving a symplectic form $L$.  We investigate what algebraic units could potentially appear as dilatations by first showing that every algebraic unit $\lambda$ appears as an eigenvalue for some integral symplectic matrix.  We then show that if $\lambda$ is real and the greatest in modulus of its algebraic conjugates and their inverses, then $\lambda^n$ is the spectral radius of an integral Perron-Frobenius matrix preserving a prescribed symplectic form $L$.  An immediate application of this is that for $\lambda$ as above, $\log\left(\lambda^n\right)$ is the topological entropy of a subshift of finite type.
\end{abstract}

\section{Introduction}\label{S:Introduction}

We recall that a self-homeomorphism $\phi$ of a surface $F$ with $\chi\left(F\right) < 0$ is called \emph{pseudo-Anosov} 
if it leaves invariant a pair of transverse, singular, measured foliations $\mathfrak{F}^s$, $\mathfrak{F}^u$ called the 
stable and unstable foliations, respectively.  Associated to such a map is an algebraic unit $\lambda_\phi$  called the \emph{dilatation} of $\phi$
which measures how the map stretches $\mathfrak{F}^s$ and shrinks $\mathfrak{F}^u$.  The dilatation encodes a variety of dynamical 
properties, for example the topological entropy of $\phi$ is $log(\lambda_\phi)$. Recently there has been a great deal of interest in 
the dilatations of pseudo-Anosov automorphisms, including a recent paper of Farb, Leininger, and Margalit which
explores connections between low dilatation pseudo-Anosovs and 3-manifolds (see~\cite{FLM}). More generally, the question of
which dilatations can be realized by some pseudo-Anosov has received attention (see for example~\cite{F} and~\cite{LT}).

There are a number of  ways to find the dilatation $\lambda_\phi$ of a pseudo-Anosov $\phi$.  By taking suitable branched
coverings,  $\lambda_\phi$  can be made to appear as the largest  root of an integral symplectic matrix.  In fact, in~\cite{P} Penner describes a symplectic pairing which is preserved by the action of $\phi$ by an integral Perron-Frobenius matrix.  This matrix encodes the action of $\phi$ on a train track $\tau$ which carries it, and the dilatation appears as the spectral radius of the matrix (for more on train tracks and pseudo-Anosovs, see~\cite{BH},~\cite{P}, and ~\cite{FM}).  Different train tracks and different pseudo-Anosovs will have different symplectic pairings associated to them.  The pairing in general may have degeneracies, but in large classes of examples the pairing is non-degenerate (and in fact a symplectic form).

  The motivation for this paper came from thinking about what algebraic units appear as spectral radii of integral symplectic 
Perron-Frobenius  matrices, and hence could potentially appear as dilatations of pseudo-Anosov automorphisms.  Additionally, we want to be able to construct these matrices to preserve a prescribed symplectic form.

Let $\lambda \in \R$ be an algebraic unit, that is, $\lambda$ is the root of a polynomial which is irreducible over the integers and of the form $p\left(t\right) = t^g + a_{g}t^{g-1} + ... + a_2t \pm 1$.  If also $| \lambda | > 1$, $\lambda$ has algebraic multiplicity $1$, and for all other roots $\omega$ of $p\left(t\right)$ we have $| \lambda + \lambda^{-1} | > | \omega + \omega^{-1} |$ we will say $\lambda$ is a \emph{Perron unit}.  From $p\left(t\right)$, we can form a self-reciprocal (or palindromic) polynomial $q\left(t\right) = t^g p\left(t\right) p\left(t^{-1}\right)$ for which $\lambda$ and $\lambda^{-1}$ are both roots.  If $\lambda$ is a Perron unit, then it is the unique largest root of $q\left(t\right)$.

We want to find Perron units which appear as the spectral radius of a symplectic Perron-Frobenius matrix.  In particular, we will prove: \\
\begin{mainTheorem}
Let $\lambda$ be a Perron unit, and let $L$ be any integral symplectic form.

Then for some $n \in \N$, $\lambda^n$ is the spectral radius of an integral Perron-Frobenius matrix which preserves the symplectic form $L$.
\end{mainTheorem}

The proof is constructive enough that it is possible to find a matrix for $\lambda$ with the assistance of a computer.




The rest of this paper is divided into three parts.  In the first part, we give a canonical form for integral symplectic matrices so that it is easy to construct a matrix preserving a given symplectic form and having a given self-reciprocal polynomial as its characteristic polynomial.  In the second part, we show how to conjugate a power of these matrices to be Perron-Frobenius.  In particular, we prove:
\\
\begin{thm}
Let $M$ be an integral matrix with a unique, real eigenvalue of largest modulus greater than 1.  Suppose also that this eigenvalue has algebraic multiplicity 1, and that $M$ preserves a symplectic form $L$.

Then $\exists n \in \N$ and $B\in\GL\left(2g\right)$ such that $B^{-1} M^n B$ is an integral, Perron-Frobenius matrix.  Furthermore, $B^{-1} M^n B$ will also preserve $L$.
\end{thm}

In the final section, we give an immediate application of some of these results to subshifts of finite type.  Given an integral Perron-Frobenius matrix, it is always possible to build a larger Perron-Frobenius matrix whose entries are all either 0 or 1.  This new matrix will have the same spectral radius as the original one, so the results above show that every Perron unit appears as the spectral radius of a such a matrix.  In fact, up to multiplaction by $t^k$, the characteristic polynomial of the new matrix is the same as the one it was built from.  We include this discussion both as a simple application and because it may also be useful in studying pseudo-Anosovs.


Although the motivation for this paper was to study pseudo-Anosov maps, there are applications of these results outside the study of surface automorphisms.  See for example~\cite{H}.  To the author's knowledge these results are unknown, though some may seem like basic facts.

\subsection{Acknowledgements}

The author would like to thank his advisor, Darren Long, for numerous helpful conversations while working on the results of this paper.  He would also like to thank Jon McCammond and Tom Howard for insightful conversations about the symplectic group and Perron-Frobenius matrices, respectively.

\section{A Canonical Form for Self-Reciprocal Polynomials}\label{S:CanonicalSympForm}

In this section, we establish a canonical form for integral matrices with self-reciprocal characteristic polynomial.  These matrices preserve a symplectic form which is standard in the sense that it arises naturally from the study of surface automorphisms.

A polynomial $p\left(t\right)$ over the integers is \emph{self-reciprocal} if its coefficients are palindromic, i.e, $p\left(t\right)$ has the form

\begin{equation}\label{E:Palindromic}
p\left(t\right) = 1 - a_2 t - a_3 t^2 - ... - a_{g+1} t^g - a_g t^{g+1} - ... - a_2 t^{2g - 1} + t^{2g}
\end{equation}


Let $\Sp$ be the symplectic group over $\mathbb{R}^{2g}$.  Up to change of basis, we may represent any non-degenerate, skew-symmetric bilinear form by either

$$
J = \left(\begin{array}{ccccc}0 & 1 &  &  & 0 \\-1 & 0 &  &  &  \\ &  & \ddots &  &  \\ &  &  & 0 & 1 \\0 &  &  & -1 & 0\end{array}\right)
$$

or

$$
K = \left(\begin{array}{cc}0 & I \\-I & 0\end{array}\right)
$$

Where $I$ represents the $g \cross g$ identity matrix.  We specify $J$ because it is the symplectic form we usually think of when considering the action of a surface automorphism on the first homology group of the surface.  We include $K$ because it is easier to work with in obtaining the results of this section.

We now define two standard forms for a matrix which has the self-reciprocal polynomial $p\left(t\right)$ as its characteristic polynomial.  We will also show that each preserves one of the standard symplectic forms above.  The first canonical form, denoted $A$ below, preserves $J$ (that is, $A^T J A = J$).

$$
A = \left(\begin{array}{ccccccccc}
0 & \hdots &  &  &  &  & \hdots & 0 & -1 \\
0 & a_2 & 0 & a_3 & \hdots & 0 & a_g & 1 & a_{g+1} \\
1 & 0  &  &  &  &  &  &  & a_2 \\
0 & 1 &  &  &  &  &  &  & 0 \\
\vdots &  & \ddots &  &  &  &  &  & a_3 \\
 &  &  &  &  &  &  &  & \vdots \\
  &  &  &  & \ddots  &  &  &  & 0 \\
   &  &  &  &  & 1 & 0 &  & a_g \\
   0 & \hdots &  &  & \hdots & 0 & 1 & 0 & 0\end{array}\right)
$$


By performing the change of basis which carries $J$ to $K$, we obtain a second canonical form, denoted $B$, which preserves $K$.

$$
B = \left(\begin{array}{cccccccc}
0 & \hdots &  &  &  &  & \hdots & -1 \\
1 &  &  &  &  &  &  & a_2 \\
& \ddots &  &  &  &  &  & a_3 \\
&  & \ddots &  &  &  &  & \vdots \\
&  &  & 1 & a_2 & a_3 & \hdots & a_{g+1} \\
&  &  &  & \ddots &  &  & 0 \\ 
&  &  &  &  & \ddots &  & \vdots \\
0 &  &  &  &  &  & 1 & 0
\end{array}\right)
$$
The proofs of this section could be considered tedious, and the uninterested reader should have no problems skipping to section~\ref{S:ChangeBasisToPF} after first reading theorem~\ref{T:AlgUnitsAreSymplectic}.
\\
\begin{lemma}\label{L:PreservesForm}
$A$ preserves the symplectic form $J$ and $B$ preserves the symplectic form $K$.
\end{lemma}

\begin{proof}
It suffices to show that $B$ preserves $K$.  Let $\{ e_1, ... , e_{2g} \}$ denote the standard basis vectors for $\mathbb{R}^{2g}$.  We note that the action of $B$ on $e_i$ is:

\begin{align*}
& B e_i  = e_{i + 1}  & \mathrm{if} \  1 \leq i \leq g  \\
& B e_i  = a_{i - g + 1} e_{g+1} + e_{i+1} & \mathrm{if} \  g + 1 \leq i \leq 2g - 1  \\
& B e_{2g}  = -e_1 + \sum_{i = 2}^{g + 1} a_i e_i  &
\end{align*}

We now show that if $<~,~>$ is the bilinear form coming from $K$, $<B e_i, B e_k>\  =\  <e_i , e_k>$.  Since this is all computational, we will do only a few cases here.  A key observation to simplify calculations is that for $1 \leq i \leq g$ we have $< e_i, e_k > \neq 0$ if and only if $k = g + i$.  In particular,  $< e_i, e_{g + 1} > \neq 0$ if and only if $i = 1$.  

First we will let $ 1 \leq i \leq g$.  Then:

$$
<B e_i, B e_k>\  =\  < e_{i+1}, B e_k>\  = 
\begin{cases}
< e_{i + 1}, e_{k + 1} > & \mbox{if } 1 \leq k \leq g \\ 
< e_{i + 1}, a_{k - g + 1} e_{g+1} > + < e_{i + 1}, e_{k+1} > & \mbox{if } g + 1 \leq k \leq 2g - 1 \\
< e_{i + 1}, -e_1 > + < e_{i + 1}, \sum_{j = 2}^{g + 1} a_j e_j >  & \mbox{if } k = 2g 
\end{cases}
$$

But checking our form $K$, we see that 

$$
<B e_i, B e_k>\  =
\begin{cases}
0 & \mbox{if } 1 \leq k \leq g \\
0 + 1 & \mbox{if } k = g + i \mbox{ and }  g + 1 \leq k \leq 2g - 1  \\
0 + 0 & \mbox{if } k \neq g + i \mbox{ and }  g + 1 \leq k \leq 2g - 1  \\
1 + 0 & \mbox{if } i = g \mbox{ and } k = 2g \\
0 + 0 & \mbox{if } i \neq g \mbox{ and } k = 2g    
\end{cases}
$$

A slightly more complicated case occurs if we let $ g+1 \leq i \leq 2g-1 $ and $ k = 2g$.  Then:

\begin{align*}
< B e_i, B e_k > & = a_{i - g + 1} <e_{g+1}, B e_{2g} > + < e_{i+1}, B e_{2g} > \\
& = a_{i - g + 1} + 0 + 0 +  \sum_{j = 2}^{g + 1} a_j < e_{i+1}, e_j > \\
& = a_{i - g + 1} - a_{i - g + 1}  \\
& = 0
\end{align*}

The other cases are not more difficult than the two above.

\end{proof}

Now we will show that $A$ and $B$ both have characteristic polynomials of form (\ref{E:Palindromic}).
\\
\begin{lemma}\label{L:CharPoly}
The characteristic polynomials of $A$ and $B$ are both $p\left(t\right) = 1 - a_2 t - a_3 t^2 - ... - a_{g+1} t^g - a_g t^{g+1} - ... - a_2 t^{2g - 1} + t^{2g}$.
\end{lemma}

\begin{proof}
As with the proof of lemma \ref{L:PreservesForm}, we prove our result for $B$ and the result immediately follows for $A$.

Let $B_0 = B - tI$, and let $B_{k+1}$ be the matrix obtained from $B_k$ by blocking off the first row and first column.  Then the $\left(0, 2g-k\right)$ minor of $B_k$ is $1$ for $0 \leq k < g$.  Thus we see that
\begin{equation}\label{E:1}
\det \left(B - tI\right) = 1 + a_2 \left(-t\right) + \left(-a_3\right) \left(-t\right)^2 + ... + \left(-1\right)^g a_g \left(-t\right)^{g - 1}  + \left(-t\right)^g \det{B_g}
\end{equation}

Where $B_g$ has form:
$$
B_g = \left(\begin{array}{ccccccccc}
a_2 -t & a_3 & \hdots & \hdots & a_{g+1} \\
1 & -t  &  & & 0 \\ 
& \ddots & \ddots  & & \vdots \\
& & \ddots & \ddots & \vdots \\
0 & &  & 1 & -t
\end{array} \right)
$$

Let $D_g = B_g$ and for $l \geq g$ let $D_{l-1}$ be the matrix obtained from $D_l$ by blocking off the last row and last column.  Then for $g \geq l > 2$, the $\left(0, l\right)$ minor of $D_l$ is $1$.  Thus we have:


\begin{align}
\det{B_g} & = \left(-1\right)^{g + 1} a_{g+1} + ... +  \left(-t\right)^{i} \left(-1\right)^{g + 1 - i} a_{g + 1 - i} + ... + \left(-t\right)^{g-3} \left(-1\right)^4 a_4 + \left(-t\right)^{g-2} \det{D_2} \nonumber \\
& = \left(-1\right)^{g+1} a_{g+1} + ... + \left(-1\right)^{g+1} t^i + ... + \left(-1\right)^{g+1} t^{g-3} + \left(-t\right)^{g-2} \det{D_2} \label{E:2}
\end{align}

Notice that in the equation above that if $g$ is even, then every coefficient is negative.  If $g$ is odd, every coefficient is positive.  Now,

\begin{equation}\label{E:3}
\det{D_2} = \det{ \left(\begin{array}{ccccccccc}
a_2 - t & a_3 \\
1 & -t
\end{array}\right) } = t^2 - a_2 t - a_3
\end{equation}

Now by substituting (\ref{E:3}) into (\ref{E:2}) into (\ref{E:1}), we obtain our result.
\end{proof}

Putting lemmas \ref{L:PreservesForm} and \ref{L:CharPoly} together, we have the following theorem:
\\
\begin{theorem}\label{T:AlgUnitsAreSymplectic}
Every algebraic unit is an eigenvalue of some symplectic matrix.
\end{theorem}

\begin{proof}
Let $\lambda$ be an algebraic unit with minimum polynomial $q\left(t\right) = 1 + b_2t + b_3 t^2 + ... + b_g t^{g - 1} + t^g$.  Then $t^g q\left(t\right)q\left(t^{-1}\right)$ is a self-reciprocal polynomial.  Applying  lemmas \ref{L:PreservesForm} and \ref{L:CharPoly} we obtain our result.
\end{proof}


\section{Changing Basis to be Perron-Frobenius}\label{S:ChangeBasisToPF}

We say a real matrix $M$ is \emph{Perron-Frobenius} if it has all nonnegative entries and $M^k$ has strictly positive entries for some $k \in \N$.  Such matrices have important applications in dynamical systems, graph theory, and in studying pseudo-Anosov surface automorphisms.  A key result about such matrices was proved in the early 20th century:
\\
\begin{PF}
Let $M$ be Perron-Frobenius.  Then $M$ has a unique eigenvalue of largest modulus $\lambda$.  Furthermore, $\lambda$ is real, positive, and has an associated real eigenvector with all positive entries.
\end{PF}

The eigenvalue $\lambda$ is called the \emph{spectral radius} or \emph{growth rate} of $M$.  The main purpose of this section is to find integral matrices which can be conjugated to be Perron-Frobenius.  We'd also like to do this in a way which preserves a fixed symplectic form (for example, the symplectic form $J$ from section \ref{S:CanonicalSympForm}).  In particular, we prove the following:
\\
\begin{theorem}\label{T:PFConjugation}
Let $M \in \SpL$ such that $M$ has a unique, real eigenvalue of largest modulus greater than 1.  Suppose also that this eigenvalue has algebraic multiplicity 1.

Then $\exists n \in \N$ and $B \in \GL\left(2g\right)$ such that $B^{-1} M^n B$ is a Perron-Frobenius matrix in $\SpL$.
\end{theorem}

Here we denote by $\SpL$ the group of $2g \cross 2g$ integer matrices which preserve a fixed symplectic form $L$.  When we do not care to fix a particular symplectic form, we will use the notation $\Sp$ to mean the group of symplectic linear transformations on $\Rg$.

We also obtain a similar result for integral, nonsingular matrices (see corollary~\ref{C:NoSymplectic}).


Given a matrix $M$ with a unique real eigenvalue of largest modulus greater than $1$, we will denote this eigenvalue $\lambda_M$ and its associated eigenvector $v_M$.  We will refer to $\lambda_M$ and $v_M$ as the dominating eigenvalue and dominating eigenvector, respectively.



The idea behind the proof will be to find an integral basis $\{ b_1, ..., b_{2g} \}$ for $\Rg$ such that $v_M$ is contained in the cone determined by $b_1, ... ,b_{2g}$.  We also need that if $W$ is the co-dimension $1$ invariant subspace of $M$ such that $v_M \notin W$, then $b_1, ... , b_{2g}$ all lie on the same side of $W$ as $v_M$.  To make the notion of side precise, denote by $W^+$ as the set of all vectors in $\Rg$ that can be written as $a v_M + w$ where $a \in \mathbb{R}^+$ and $w \in W$.
\\
\begin{lemma}\label{L:PFBasis}
Let $M$ be a matrix with a dominating real eigenvalue $\lambda_M$ and associated real eigenvector $v_M$.  Say $\{ b_1, ..., b_{2g} \}$ is a basis for $\Rg$ such that $b_1, ..., b_{2g} \in W^+$ and $v_M$ is contained in the interior of the cone determined by $b_1, ..., b_{2g}$.  

Then for some $n \in \mathbb{N}$, $M^n$ has all positive entries after changing to the basis above.
\end{lemma}

\begin{proof}
Since we can replace $M$ by $M^2$ if necessary, we may assume $\lambda_M$ is positive.  Let $\lambda_2, ..., \lambda_n$ be the other eigenvalues of $M$ and let $v_M, v_2, ..., v_{2g}$ be a Jordan basis for $M$ (i.e, a basis in which the linear transformation represented by $M$ is in Jordan canonical form).  Note that $v_2, ..., v_{2g}$ span $W$.

Consider a Jordan block associated to some eigenvalue $\lambda_i$ of $M$:

$$
J_i = \left(\begin{array}{ccccccccc}
\lambda_i & 1 & & \\
& \lambda_i & \ddots & & \\
& & \ddots & 1 & \\
& & & \lambda_i & \\
\end{array} \right)
$$

The definition of matrix multiplication guarantees that each entry of $J_i^k$ will be a polynomial in $\lambda_i$.  Each diagonal entry will equal $\lambda_i^k$ and every other entry of $J_i^k$ will have degree strictly less than $k$.  Thus we see that if $v_j$ is a Jordan basis vector corresponding to the eigenvalue $\lambda_i$ we get $\displaystyle \frac{J_i^k v_j}{\lambda_M^k} \goesto ~ 0$ as $k \goesto ~ \infty$, which implies:

\begin{equation}\label{E:JordanGrowth}
\displaystyle \frac{M^k v_j}{\lambda_M^k} \goesto ~ 0 ~ \mathrm{as} ~ k \goesto ~ \infty
\end{equation}




Since $v_M$ is in the interior of the cone determined by $b_1, ..., b_{2g}$, for positive real scalars $a_1, ..., a_{2g}$ we have $v_M = a_1 b_1 + ... + a_{2g} b_{2g}$.  Furthermore, since $b_i \in W^+$, for some positive real scalar $c_i$ and $w \in W$ we have $b_i = c_i v_M + w$.  Since $w$ may be expressed as a linear combination of $v_2, ..., v_{2g}$, we see that $\displaystyle \frac{M^k b_i}{\lambda_M^k} \goesto ~ c_i v_M$ as $k \goesto ~ \infty$ by (\ref{E:JordanGrowth}).  Rewriting $v_M$ and $w$ as (real) linear combinations of $b_1, ..., b_{2g}$, we see that for $k$ large enough $M^k b_i$ is a positive linear combination of $b_1, ..., b_{2g}$.  Hence, $M^k$ has all positive entries in the basis $b_1, ..., b_{2g}$.
\end{proof}

The last paragraph of the proof above also gives us a quick but important corollary.  We will use $|| \cdot ||$ to denote the standard Euclidean norm.
\\
\begin{corollary}\label{C:SuckedIn}
Let $M$ as in lemma \ref{L:PFBasis} and $v \in W^+$.  Then the distance between $\displaystyle \frac{M^k v}{|| M^k v ||}$ and $\displaystyle \frac{v_M}{|| v_M ||}$ approaches $0$ as $k \goesto ~ \infty$.
\end{corollary}

Our goal is now to construct a matrix $B \in \SpL$ such that the columns of $B$ form a basis satisfying the hypotheses of lemma \ref{L:PFBasis}.  The idea will be to construct a set of symplectic basis vectors which define a very narrow cone, and then apply a slightly perturbed symplectic isometry of $\Sg$ to move that cone into the correct position.

A symplectic linear transformation $\tau$ is a (symplectic) transvection if $\tau \neq 1$, $\tau$ is the identity map on a codimension 1 subspace $U$, and $\tau v - v \in U$ for all $v \in \Rg$.  Geometrically, a tranvection is a shear fixing the hyperplane $U$.  A symplectic transvection preserving the symplectic form $J$ can be written 

$$\tau_{u, a} v = v + a J\left(v, u\right) u$$

for some scalar $a$ and vector $u \in \Rg$.  Note that the fixed subspace is $<u>^\perp$ and that it contains $u$.  $\Sp$ is generated by transvections (see \cite{Grove}).  If we wish to preserve a symplectic form $L$ different from $J$, simply replace $J$ with $L$ in the formula.




Let $u \in \Rg$ be the vector $\left(-1, 1, ... , -1, 1\right)$ and set $a = 1$.  Let $e_1, ..., e_{2g}$ be the standard basis for $\Rg$.  Notice $J\left(e_i, u\right) = 1$, so $\tau_{u, 1} e_i = e_i + u$.  Thus, in matrix form:

$$
\tau_{u, 1} = \left(
\begin{array}{ccccc}
0 & -1 & & -1 & -1 \\
1 & 2 & & 1 & 1 \\
\vdots & \vdots & \ddots  & \vdots & \vdots \\
-1 & -1 & & 0 & -1 \\
1 & 1 & & 1 & 2 \\
\end{array}
\right)
$$

Composing this with transvections $\tau_{e_k, 2}$ with $k$ even, we get the symplectic matrix

$$
A = \left(
\begin{array}{ccccc}
2 & 3 & & 1 & 1 \\
1 & 2 & & 1 & 1 \\
\vdots & \vdots & \ddots  & \vdots & \vdots \\
1 & 1 & & 2 & 3 \\
1 & 1 & & 1 & 2 \\
\end{array}
\right)
$$

This matrix preserves the symplectic form $J$, and is also Perron-Frobenius.  In fact, we can find such a matrix for any integral symplectic form:
\\
\begin{lemma}\label{L:SPFExists}
There is a Perron-Frobenius matrix in $\SpL$ for any integral symplectic form $L$.
\end{lemma}

\begin{proof}
Non-degeneracy of $L$ guarantees that there is $u \in \Q^{2g}$ such that $L\left(e_i, u\right) = 1$ for every basis vector $e_i$.  Let $w = \left(1, 1, ..., 1\right) \in \Q^{2g}$, and notice that $L\left(u, w\right) = -2g$.  Then $\tau_{u, a}e_i = e_i + a u$ for for $a$ very large we have that $\tau_{u, a} e_i$ is close to $c u$ for some $c \in \N$.  Now by continuity, $L\left(\tau_{u, a}e_i, w \right) = l < 0$ and for $b \in \N$ we have $\tau_{w, -b}\tau_{u, a} e_i = \tau_{u, a} e_i - b l w$.  Thus for $b$ large enough, $\tau_{w, -b} \tau_{u, a} e_i$ is a rational vector with positive entries for all $i$.  This transformation has Perron-Frobenius matrix representation.  If it is not integral, we can adjust the values of $a$ and $b$ to clear denominators.
\end{proof}

Let $\Ug$ denote the group of unitary linear transformations of $\C^g$.  Equivalently, we can think of the unitary group as a group of matrices: $\Ug = \{ M | M \in \mathrm{GL}\left(g, \C \right), M^*M = I \}$ where $M^*$ denotes the conjugate transpose of $M$.

We identify $\Ug$ with a subgroup of $\mathrm{GL}\left(2g, \R\right)$ as follows:  Let $M \in \Ug$.  Replace every entry $m = r e^{\left(i \theta \right)} \in \C$ in $M$ by the scaled $2 \cross 2$ rotation matrix 
$R = \left(
\begin{array}{ccccc}
r\cos\left(\theta\right) & -r\sin\left(\theta\right) \\
r\sin\left(\theta\right) & r\cos\left(\theta\right)
\end{array}
\right)$  
We now can consider $\Ug$ as a group of real matrices acting on $\Rg$.  Notice that if $m \mapsto R$, then $\bar{m} \mapsto R^T$.  Thus, if $M =\left[m_{i,j}\right] \in \Ug$ is identified with $N = \left[R_{i,j}\right]$, we have $M^*M = \left[\bar{m}_{i,j}\right]^T \left[m_{i,j}\right] \mapsto \left[R_{i,j}^T\right]^T \left[R_{i,j}\right] = N^T N = I$.  Hence with this identification $\Ug$ is a subgroup of the real orthogonal group $\mathrm{O}\left(2g\right)$ (in fact it is a subgroup of $\mathrm{SO\left(2g\right)}$).

Notice that the symplectic form $J$ gets identified with the complex matrix 

$$
\left(\begin{array}{ccccc}
-i & & \\
& \ddots & \\
& & -i \\
\end{array} \right)
$$

which is in the center of $\Ug$.  Then if $M \in \Ug$ we have $M^* J M = J$, and thus $\Ug$ is a subgroup of $\Sp$.  Below is a more powerful result which is proved in \cite{MS} as lemma 2.17.
\\
\begin{lemma}\label{L:UnitaryIntersection}
$\Sp \cap \Og = \Ug$
\end{lemma}

We also need the following fact:
\\
\begin{lemma}\label{L:UnitaryTransitive}
The unitary group $\Ug$ acts transitively on $\Sg \subseteq \Rg$.
\end{lemma}

\begin{proof}
The $\Sg$ sphere can be thought of as all vectors in $\C^g$ having unit length.  Let $v \in \Sg$ and $\{e_1, ... , e_g\}$ be the standard basis for $\C^g$.  Using the Gram-Schmidt process, we can extend $v$ to an orthonormal basis $\{v, v_2, ..., v_g\}$ for $\C^g$.  Then the change of basis matrix is in $\Ug$ and sends $e_1$ to $v$.
\end{proof}


At one point during the proof of our main theorem, it will become important to know that $\Spq$ is dense in $\Sp$.  This follows quickly from the Borel Density Theorem, but we include an elementary proof.
\\
\begin{lemma}\label{L:SymplecticDense}
$\Spq$ is dense in $\Sp$.
\end{lemma}

\begin{proof}
Let $M' \in \mathrm{Sp}\left(2g, \R, J\right)$.  Perturb the entries of $M'$ by a small amount to obtain a matrix $M$ with rational entries.  We will systematically modify the columns $a_1, b_1, ..., a_g, b_g$ of $M$ to form a new $M$ which preserves $J$ and still differs from $M'$ by a small amount.  Here for convience we let $<~,~>$ denote the symplectic form given by $J$.

We iterate the following procedure for each pair of columns $a_i$, $b_i$, starting with $a_1$, $b_1$.  First, say $<a_i, b_i> = 1 + \eta_i$ where $\eta_i$ is a small, rational number (its magnitude depends on the size of the perturbation of $M'$).  Replace $a_i$ with $\displaystyle \frac{a_i}{1 + \eta_i}$, so that now $<a_i, b_i> = 1$.  Now we modify each pair of columns $a_j, b_j$ with $j > i$.  Set $\eps_{i,j} = < a_i, a_j >$ and $\delta_{i,j} = < b_i, a_j >$.  Replace $a_j$ with $a_j - \eps_{i,j} b_i - \delta_{i,j} a_i$, so that now $< a_i, a_j > = <b_i, a_j > = 0$.  Note that $\eps_{i,j}$ and $\delta_{i,j}$ are also small rational numbers.  Now modify $b_j$ by a similar procedure, so that $< a_i, b_j > = <b_i, b_j > = 0$.

Now repeat the procedure with the columns $a_{i+1}$, $b_{i+1}$.  After modifying every column we obtain a new $M$ which is in $\mathrm{Sp}\left(2g, \Q, J\right)$.  Furthermore, since at each stage the modifications to the columns are small, $M$ is still close to $M'$.
\end{proof}

We're now ready to prove theorem~\ref{T:PFConjugation}.  Throughout we will use the notation that if $v \in \Rg \setminus \{0\}$ then $\hat{v}$ denotes the normalization $v / ||v|| \in \Sg$.  If $M$ is a matrix with no zero columns, then $\hat{M}$ will denote the matrix obtained by normalizing each of the columns.

\begin{proof}[proof of theorem~\ref{T:PFConjugation}]
Let $M \in \SpL$ with dominating real eigenvalue $\lambda$ and associated eigenvector $v_M$.  Let $W$ be the co-dimension $1$ invariant subspace of $M$ with $v_M \notin W$, and $W^+$ the component of $\Rg \setminus W$ containing $v_M$.  Set $\eps$ to be the minimal distance in $\Sg$ from $\hat{v}_M$ to $W \cap \Sg$.  Then by lemma \ref{L:SPFExists} and corollary \ref{C:SuckedIn}, there exists $n \in \mathbb{N}$ and $A \in \SpL$ such that $A$ is Perron-Frobenius and the convex hull $H$ of the columns of $\widehat{A^n}$ has diameter less than $\eps$ (here we take $H \subseteq \Sg$ and measure distance in $\Sg$).

Let $\nu$ be in the interior of $H$.  Since $\Ug$ acts transitively on $\Sg$ (lemma \ref{L:UnitaryTransitive}), there is $S \in \Ug$ such that $S\nu = \hat{v_M}$.  As a real linear transformation, $S$ is orthogonal and hence $\mathrm{diam}\left(H\right) = \mathrm{diam}\left(S\left(H\right)\right)$.  Thus the columns of $S\widehat{A^n}$ are contained in $W^+$.  $\Ug$ is a subgroup of $\Sp$ (lemma~\ref{L:UnitaryIntersection}), so $S \in \Sp$.  Furthermore, by lemma \ref{L:SymplecticDense} we may perturb $S$ slightly so that now $S \in \SpLq$.  Set $B' = SA^n$, note $B' \in \SpLq$.  Scale $B'$ by an integer $\alpha$ so that $B = \alpha B' $ is a nonsingular, integral matrix.

Set $d = \det B$.  Then $B^{-1} = \frac{1}{d} C$, where $C$ is the adjugate of $B$.  In particular, $C$ is integral.

Consider the projection map $\mathrm{SL}\left(2g, \Z\right) \rightarrow \mathrm{SL}\left(2g, \Z / d\Z \right)$.  Since $\mathrm{SL}\left(2g, \Z / d\Z \right)$ is finite, for some $m \in \N$ we have $M^m$ in the kernel of this map.  Hence, we can write $M^m = I + d \Lambda$ for some integral matrix $\Lambda$.  Putting this together, we have:

\begin{align*}
B\inv M^m B & = \frac{1}{d} C \left(I + d \Lambda \right) B \\
& = I + C \Lambda B 
\end{align*}

In particular, $B\inv M^m B$ is integral.  By construction, the columns of $B$ give a basis satisfying the conditions of lemma \ref{L:PFBasis}, so for large enough $k \in \N$ we have $B\inv M^{m k} B$ is Perron-Frobenius and integral.  Furthermore $B\inv M^{m k} B$ is symplectic since $B$ is a scaled symplectic matrix.

\end{proof}

Using  theorems~\ref{T:AlgUnitsAreSymplectic} and~\ref{T:PFConjugation}, we can prove our main result, which we restate here:
\\
\begin{theorem}
Let $\lambda$ be a Perron unit, and let $L$ be any integral symplectic form.

Then for some $n \in \N$, $\lambda^n$ is the spectral radius of an integral Perron-Frobenius matrix which preserves the symplectic form $L$.
\end{theorem}

\begin{proof}
Using the canonical form of section~\ref{S:CanonicalSympForm}, we can build a matrix $M \in \SpJ$ with $\lambda$ its spectral radius.  For some $B' \in \GL\left(2g, \Q \right)$ we have $\left(B'\right)^T J B' = L$.  Scale $B'$ by an integer $\alpha$ so that $B = \alpha B'$ is integral.  Now proceeding with the argument at the end of the proof for theorem~\ref{T:PFConjugation}, we get that $B\inv M^r B \in \SpL$.  Now we can apply theorm~\ref{T:PFConjugation} to obtain our result.

\end{proof}

We end this section by noting that if the matrix $M$ is not symplectic, we can modify the hypotheses slightly to achieve a result similar to theorm~\ref{T:PFConjugation}.  The proof uses similar ideas, but is actually significantly easier.
\\
\begin{corollary}\label{C:NoSymplectic}
Let $M$ be an integral, nonsingular matrix with a unique, real eigenvalue of largest modulus greater than 1.  Suppose also that this eigenvalue has algebraic multiplicity 1.  

Then $\exists n \in \N$ such that $M^n$ is conjugate to an integral Perron-Frobenius matrix.
\end{corollary}

\begin{proof}
Let $\delta = \det{M}$, and pick a $B' \in \SL \left(r, \Q\right)$ such that the columns of $B'$ satisfy the conditions of lemma \ref{L:PFBasis}.  Choose $\alpha \in \Z$ such that $\widetilde{B} = \alpha B'$ has integer entries and $\delta$ divides every entry of $\widetilde{B}$.  Assuming we also chose $\alpha$ to be large, we may set $B = \widetilde{B} + I$ and the columns of $B$ will still satisfy lemma \ref{L:PFBasis}.

Consider $d = \det{B}$.  Calculating the determinant by cofactor expansion, we see that $d = \left(sum~of~terms~divisible~by~ \delta \right) + 1$.  In particular, $\delta$ is relatively prime to $d$, so $M$ has a projection to $\GL\left(r, \Z / d\Z \right)$.  We now raise $M$ to a power $m$ so that $M^m = I + d \Lambda$ and proceed with the argument of theorem \ref{T:PFConjugation}.


\end{proof}

\section{Subshifts of Finite Type}\label{S:SoFT}

We will now apply the previous two sections to symbolic dynamics, in particular to subshifts of finite type.



Let $M$ be an $n \cross n$ matrix of 0's and 1's.  Let $A_n = \{ 1, 2, ..., n \}$, and form $\Sigma_n = A_n \cross \Z$.  We can think of $\Sigma_n$ as the set of all bi-infinite sequences in symbols from $A_n$, and we endow it with the product topology.  Now we form a subset $\Lambda_M \subseteq \Sigma_n$ by saying $\left(s_i\right) \in \Lambda_M$ if the $s_i, s_{i + 1}$ entry of $M$ is equal to $1$ for all $i$.  We can think of the $i, j$ entry of $M$ as telling us whether it is possible to transition from state $i$ to state $j$.  Now let $\sigma$ be the automorphism of $\Lambda_M$ obtained by shifting every sequence one place to the left.  The dynamical system $\left(\Lambda_M, \sigma \right)$ is called a \emph{subshift of finite type}, and can be thought of as a zero-dimensional dynamical system.  These dynamical systems have relatively easy to understand dynamics and are often used to model more complicated systems (for example, pseudo-Anosov automorphisms).

Let $M = \left[ m_{i,j}\right]$ be a square matrix with nonnegative, integer entries.  We form a directed graph $G$ from $M$ as follows.  $G$ has one vertex for each row of $M$.  Then connect the $i$-th vertex to the $j$-th vertex by $m_{i,j}$ edges, each directed from vertex $i$ to vertex $j$.  We call $M$ the \emph{transition matrix} for $G$.  If $M$ is Perron-Frobenius, then the graph $G$ will be strongly connected and the $i, j$-th entry of $M^k$ represents the number of paths of length $k$ from vertex $i$ to vertex $j$.  The spectral radius $\lambda$ of $M$ can be interpreted as the growth rate of the number of paths of length $k$ in $G$, i.e. $\displaystyle \lim_{k \rightarrow \infty} \frac{M^k}{\lambda^k} = P \neq 0$.

We now show how to go from an integral Perron-Frobenius matrix $M$ to another matrix with the same spectral radius whose entries are all $0$ or $1$.  This construction can also be found in \cite{F}.  Given a directed graph $G$ with Perron-Frobenius transition matrix $M$, label the edges of $G$ as $e_1, ..., e_n$ and the vertices $v_1, ..., v_m$.  From $G$, we form a directed graph $H$ as follows:  the vertex set  $w_1, ..., w_n$ of $H$ is in 1 - 1 correspondence with the edge set of $G$ ($w_i \leftrightarrow e_i$).   If the edge $e_i$ terminates at the vertex from which $e_j$ emanates, then we place an edge in $H$ from $w_i$ to $w_j$.  Let $N$ be the transition matrix of $H$.  Note that by construction, every entry of $N$ is either a $0$ or a $1$.
\\



A subgraph of a graph G is a \emph{cycle} if it is connected and every vertex has in and out valence $1$.  If $M$ is a transition matrix for $G$, it is possible to reformulate the calculation of the characteristic polynomial $p\left(t\right) = \det\left(t I - M \right)$ in terms of cycles in $G$ (see~\cite{BrH}):
\\
\begin{lemma}\label{L:CharPolyOfGraph}
Let $G$ be a graph with transition matrix $M$.  Denote by $\mathbf{C_i}$ the collection of all subgraphs which have $i$ vertices and are the disjoint union of cycles.  For $C \in \mathbf{C_i}$, denote by $\#\left(C\right)$ the number of cycles in $C$.  Then the characteristic polynomial $p\left(t\right) = \det\left(t I - M \right)$ is

$$p\left(t\right) = t^m + \sum_{i = 1}^m c_i t^{m - i}$$
where $m$ is the number of vertices in $G$ and
$$c_i = \sum_{C \in \mathbf{C_i}} \left(-1\right)^{\#\left(C\right)}$$

\end{lemma}

Using this formula, we can prove that the characteristic polynomial of $N$ (as above) has a nice form, and in particular that the spectral radius of $N$ is the same as the spectral radius of $M$.
\\
\begin{theorem}
Let $M$ be the transition matrix for a graph with $m$ vertices and let $N$ be an $n \cross n$ matrix of 0's and 1's built from $M$ by the construction above.  

Then if $p\left(t\right) = \det\left(t I - M \right)$ is the characteristic polynomial of $M$, the characteristic polynomial of $N$ is $q\left(t\right) = t^{n - m} p\left(t\right)$
\end{theorem}

\begin{proof}
Let $G$ be the graph associated to $M$, and $H$ the graph associated with $N$.  Order the vertices of $G$, and for each vertex $v$ fix a lexicographic order of (\emph{in-edge}, \emph{out-edge}) pairs of edges incident to $v$.  Let $\mathbf{D_i}$ be the collection of subgraphs of $H$ which can be written as a union of disjoint cycles with $i$ total vertices.  For $D \in \mathbf{D_i}$, there is a canonical projection of $D$ to a collection of paths in $G$ (using the fact that vertices in $H$ come from edges in $G$).  Let $\mathbf{D_i^*}$ be the subset of $\mathbf{D_i}$ containing those disjoint unions of cycles in $H$ which do not project to a disjoint union of cycles in $G$.  We will show that there is a bijection between elements of $\mathbf{D_i^*}$ having an odd number of components and elements of $\mathbf{D_i^*}$ having an even number of components.

Let $D \in \mathbf{D_i^*}$ and say $D$ has an odd number of components.  Call $C$ its projection to a collection of paths in $G$.  Since $C$ is not a disjoint union of cycles, there must be vertices of $G$ that are either visited by two different paths in $C$ and/or are visited twice by the same path.  Choose $v$ to be the minimal such vertex in the ordering of vertices of $G$, and note that $v$ must have in-valence and out-valence both of at least $2$.  Choose two in/out-edge pairs, $\left(e, f\right)$ and $\left(e', f'\right)$, such that each pair occurs in some path in $C$ and so that they are minimal among such pairs in the ordering of edges incident to $v$.  Note that $D$ contains vertices in $H$ corresponding to $e, e', f, f'$ and must contain edges from $e$ to $f$ and from $e'$ to $f'$.  Build $D' \in \mathbf{D_i^*}$ by letting $D'$ have the same vertex collection as $D$, but instead of containing edges from $e$ to $f$ and from $e'$ to $f'$ it contains edges from $e$ to $f'$ and $e'$ to $f$ (call this operation an \emph{edge swap}).

If the pairs $\left(e, f\right)$ and $\left(e', f'\right)$ are both part of the same cycle in $D$, then $D'$ will have one more component than $D$.  If they are part of two different cycles, then $D'$ will have one less component.  In either case, $D'$ has an even number of components and we have constructed a well-defined map from elements of $\mathbf{D_i^*}$ having odd components to elements having even components.  Note also that the projection $C'$ of $D'$ still visits $v$ twice, and contains in/out-edge pairs $\left(e, f'\right)$ and $\left(e', f\right)$.  Thus we can define the inverse of this map in exactly the same way, and hence we have a bijection.

Because of the bijection we built above, we see that disjoint unions of cycles in $\mathbf{D_i^*}$ cancel out when $q\left(t\right)$ when it is computed using lemma~\ref{L:CharPolyOfGraph}.  Elements of $\mathbf{D_i} \setminus \mathbf{D_i^*}$ are in bijective correspondence with cycles in $\mathbf{C_i}$, so we get our conclusion.
\end{proof}


Finally, we have:
\\
\begin{theorem}
Let $\lambda$ be a Perron unit.  Then there is $k \in \N$ such that $\log\left(\lambda^k\right)$ is the topological entropy of some subshift of finite type. 
\end{theorem}

This follows directly from theorems~\ref{T:PFConjugation},~\ref{T:AlgUnitsAreSymplectic}, and comments of Fathi, Laudenbach, and Po\'earu on subshifts of finite type (see~\cite{FS}).


Department of Mathematics,\\ University of California, Santa Barbara\\ Santa Barbara,
CA 93106.\\
\noindent Email:~rackermann@math.ucsb.edu\\[\baselineskip]

\end{document}